\documentclass{amsart}
\usepackage{color}
\usepackage{amsthm,amsfonts,amsmath,amscd}
\usepackage{enumitem}
\newtheorem{corollary}{Corollary}
\newtheorem{theorem}{Theorem}
\newtheorem{lemma}{Lemma}
\newtheorem{proposition}{Proposition}
\newtheorem*{Conj}{Nikiel's Conjecture}
\theoremstyle{definition}
\newtheorem{definition}{Definition}
\newtheorem{example}{Example}
\newtheorem{remark}{Remark}

\usepackage[right]{lineno}
\usepackage{color,soul}

\begin{document}

\title[Continuous images of orderable compacta]{Cardinal functions on continuous images of orderable compacta and applications}

\author{Ahmad Farhat}
\address{Affiliation at time of research: Mathematical Institute, University of Wroc\l{}aw, pl. Grunwaldzki 2/4, 50-384 Wroc\l{}aw, Poland}
\address{Current affiliation: Chatham Financial Corp., ul. Rakowicka 7, 31-511 Krak\'ow, Poland}
\email{ahmad.farhat@math.uni.wroc.pl}


\begin{abstract}
 The class of Hausdorff spaces that are continuous images of compact orderable spaces is studied by analyzing the relationship between the elements of this class and compact orderable spaces in a back-and-forth fashion. 
 Structure results for this class are then obtained, as well as continuum-theoretic embedding results.
 Applications to Boolean algebras are also demonstrated, specifically concerning the relationship between interval algebras and pseudo-tree algebras.
\end{abstract}

\maketitle


\section*{Introduction}

Let $X$ be a Hausdorff space that is a continuous image of a compact linearly ordered topological space, $\mathcal{P}$ be a cardinal invariant, and $\mathcal{A}$ a topological property. We consider the following back-and-forth scheme:
\begin{enumerate}
 \item \label{innab} If $X$ satisfies $\mathcal{P}$, is it possible to ``pull-back'' $\mathcal{P}$ from $X$ to some compact ordered space $K$ which maps continuously onto $X$?
 \item If (1) is possible, and we know that every compact ordered space that satisfies $\mathcal{P}$ must also satisfy $\mathcal{A}$, is it true then that $X$ also satisfies $\mathcal{A}$?
\end{enumerate}

The first systematic study of (\ref{innab}) was undertaken by Mardesic and Papic in \cite{MP}; 
a survey of the results obtained there recently appeared in section 5 of Mardesic's survey article \cite{Msurvey}.
We start this paper by expanding on the line of research of \cite{MP},
studying classical as well as recent cardinal invariants of an order-theoretic nature, the Noetherian types. 
Afterwards, having established relations between the properties of compact ordered spaces and of spaces they map continuously onto, we study the push-forward part of the scheme, obtaining several structure results for the class of Hausdorff spaces that are continuous images images of ordered compacta.
We then build on this to prove an embedding result, demonstrating part of the interplay between the continuum-theoretic and order-theoretic aspects of our study.
In the final part of the paper, we outline applications to Boolean algebras, specifically concerning isomorphisms between pseudo-tree algebras and subalgebras of interval algebras.
The classes of interval algebras and tree algebras are two of the more interesting classes of Boolean algebras when no extra set-theoretic axioms are assumed.

We recall at this point that one the late M.E. Rudin's breakthroughs was a proof of Nikiel's conjecture.

\begin{Conj}[M.E. Rudin \cite{Ru}] 
 A compact Hausdorff space is monotonically normal if and only if it is a continuous image of a compact linearly ordered space.
\end{Conj}

Thus, all of our results on Hausdorff spaces that are continuous images of compact ordered spaces can be formulated for monotonically normal compacta, and we use the corresponding terminology interchangeably.


\section{Preliminaries}

\subsection{Notation} 

We assume that all spaces are infinite and Hausdorff, and that all maps are continuous. A compactum is a compact space, and a continuum is a connected compactum. An ordered space which is also a continuum is called an \emph{arc}. For a space $X$, $2^X$ denotes the hyperspace of all non-empty closed subsets of $X$ equipped with the Vietoris topology. When $\mathcal{F}$ is a family of subsets of a space, $\mathcal{F}$ is always taken to be ordered by reverse inclusion. The closed unit interval $[0,1]$ will be denoted by $I$.

\subsection{Cardinal invariants}

We recall the definitions of some topological cardinal invariants. For a space $X$ and a point $x\in X$,
\begin{itemize}
 \item $w(X) = \min\{|B| : B$ is a base for $X\}$, the \emph{weight} of $X$;
 \item $\chi(X, x) = \min\{|B| : B$ is a local base for $X$ at $x\}$, the \emph{character} of $X$ at $x$;
 \item $\chi(X) = \sup\{\chi(X,x) : x\in X\}$, the \emph{character} of $X$;
 \item $d(X) = \min\{|D| : D$ is a dense subset of $X\}$, the \emph{density} of $X$;
 \item $c(X) = \sup\{|U| : U$ is a disjoint family of non-empty open subsets of $X\}$, the \emph{cellularity} of $X$;
 \item $\pi (X) = \min\{|B| : B$ is a $\pi$-base for $X\}$, the $\pi$\emph{-weight} of $X$;
 \item $\pi\chi (X, x) = \min\{|B| : B$ is a local $\pi$-base for $X$ at $x\}$, the $\pi$\emph{-character} of $X$ at $x$;
 \item $\pi\chi(X) = \sup\{\pi\chi(X,x) : x\in X\}$, the $\pi$\emph{-character} of $X$;
 \item $t(X,x)=\sup\big\{\min\{|Z|:Z\subseteq Y\ \wedge\ x\in {\rm cl}(Z)\}:Y\subseteq X\ \wedge\ x\in {\rm cl}(Y)\big\}$, the \emph{tightness} of $X$ at $x$;
 \item $t(X)=\sup\{t(X,x):x\in X\}$, the \emph{tightness} of $X$.
\end{itemize}

\begin{definition}
 Let $(P,<)$ be partially ordered set. Given a cardinal $\kappa$, $P$ is said to be $\kappa^{op}$-like if every bounded subset of $P$ has cardinality strictly less than $\kappa$.
\end{definition}

We list the definitions of some cardinal functions of an order-theoretic nature. As mentioned, any family of subsets of a space is always ordered by reverse inclusion. For a space $X$, a point $x\in X$, and a subset $A\subset X$,

\begin{itemize}
 \item $Nt(X)= \min\{\kappa : X$ has a $\kappa^{op}$-like base$\}$, the \emph{Noetherian type} of $X$;
 \item $\pi Nt(X)= \min\{\kappa : X$ has a $\kappa^{op}$-like $\pi$-base$\}$, the \emph{Noetherian} $\pi$\emph{-type} of $X$;
 \item $\chi Nt(X,A) = \min\{\kappa : X$ has a $\kappa^{op}$-like local base at $A\}$, the \emph{local Noetherian type} of $X$ at $A$;
 \item $\chi Nt(X,x)= \chi Nt(X,\{x\})$, the \emph{local Noetherian type} of $X$ at $x$;
 \item $\chi Nt(X)= \sup\{\chi Nt(X,x) : x\in X\}$, the \emph{local Noetherian type} of $X$;
\end{itemize}

\subsection{Ordered spaces and their images}

A triplet $(K,<,\tau)$, where $<$ is a linear order on the set $K$ and $\tau$ is the open-interval topology associated with $<$, is called a \emph{linearly ordered topological spaces}, or, in short, an \emph{ordered space}. A space is called orderable if it can be equipped with a linear order such that the open-interval topology and the original topology coincide.

In the metric category, the distinguished orderable compacta are the Cantor set $\mathcal{C}$ and the unit interval $I=[0,1]$. 
Classical results establish that any metric compactum is an image of $\mathcal{C}$, and that any separable orderable continuum is homeomorphic to $I$. 
In the non-metric category, the distinguished separable orderable compactum is the split interval, also called the double-arrow space. 
The split interval is the space obtained by equipping $(I\times 2) \setminus \{(0,0), (1,1)\}$ with the lexicographic order topology. 
As an analogue of the classical result that each metric compactum is an image of the Cantor set, a folklore result states that each separable monotonically normal compactum is an image of the split interval. 

\begin{definition}\label{def_split}
 For a function $\varLambda: I\to \omega$, let $I_\varLambda = \bigcup_{x\in I}\{x\}\times\{0,1,\cdots, \varLambda(x)\}$, and equip $I_\varLambda$ with the lexicographic order topology. If $S\subseteq [0,1]$, let $I_S = I_{\chi_S}$, where $\chi_S$ is the characteristic function on $S$. 
\end{definition}

It is clear that $I_{\mathbb{Q}\cap(0,1)}$ is homeomorphic to the Cantor set $\mathcal{C}$, and that $I_{(0,1)}$ is homeomorphic to the split interval.

\subsection{Reduced maps}

Let $(K,<)$ be an ordered set. A \emph{jump} in $K$ is a set $\{a,b\}\in [K]^2$ such that there is no point of $K$ between $a$ and $b$. A subset $C\subseteq K$ is said to be \emph{convex} if whenever $a,b,c\in K$ are such that $a,b\in C$ and $a<c<b$, then $c\in C$. An \emph{order component} of a subset $A\subseteq K$ is a subset $C\subseteq K$ which is maximal with respect to the properties ``$C \subseteq A$,'' and ``$C$ is convex.'' 

For a set $A$, let $\mathcal{P}(A)$ denote the power set of $A$.

\begin{definition}
 For $f: X\to Y$ a map between two spaces, the function $f^\sharp:\mathcal{P}(X)\to \mathcal{P}(Y)$ is defined by $f^\sharp(A)=Y\setminus f(X\setminus A)$ for all $A\subseteq X$.
\end{definition}

\begin{definition}
 Let $K$ be an ordered space, and $X$ and $Y$ be spaces.
\begin{itemize}
 \item A map $f: K\to X$ is said to be \emph{order-light} if for each $x\in X$, each order component of $f^{-1}(x)$ is degenerate.
 \item A surjective map $f: X\to Y$ is said to be \emph{irreducible} if $f(A)\neq Y$ for every proper closed subset $A\subseteq X$. Equivalently, $f$ is irreducible iff for each non-empty open subset $U\subseteq X$, $f^\sharp(U)\neq \emptyset$.
 \item  A map $f: K\to X$ is said to be \emph{reduced} if $f$ is both order-light and irreducible.
\end{itemize}
\end{definition}

The following lemma, due to Mardesic and Papic, and independently to Treybig, is the main tool we will be using throughout the paper.

\begin{lemma}[\cite{MP}, \cite{Tr}]\label{main_lm}
 Let $X$ be an image of an ordered compactum. Then there are an ordered compactum $K$ and a reduced map $f:K\to X$.
\end{lemma}

The proof of the lemma is more or less straightforward, and as the procedure of the proof itself will be needed in later sections, we describe it presently. First, Zorn's lemma is used to ensure irreducibility. Second, if $g:K'\to X$ is a map from an ordered compactum $K'$ onto $X$, then identifying the order components of fibers of $g$ yields an ordered compactum $K$ and an order-light map $f: K\to X$. 

The advantage of having a reduced map $f:K\to X$ is that it is in a sense minimal. First, there are no redundant open sets in the domain of $f$, and second, requiring $f$ to be order-light ensures that it does not collapse jumps (as does, for example, the standard map from the split interval onto $I$). Thus, we might heuristically expect greater similarity between the properties of $X$ and $K$ when $f$ is reduced. Indeed, this is the intuition that underlies much of Section \ref{secPullPush} below.


\section{Pull-back and push-forward of topological properties}\label{secPullPush}

\subsection{Pull-back}

We now study part (\ref{innab}) of the scheme mentioned in the introduction. Points (\ref{11}) and (\ref{22}) of Theorem \ref{main_thm} were proved in \cite{MP}, where their proofs occupy the bulk of the paper. Point (\ref{7}) also appeared in a paper of Milovich (Lemma 2.29 of \cite{Mi1}), and it holds in general for irreducible maps. In both cases, we include our concise proofs for convenience. Theorem \ref{main_thm} is followed by a series of examples showing that, wherever they may have appeared in the theorem, inequalities can be strict.
                                   
\begin{theorem}\label{main_thm}
 Let $X$ be an image of an ordered compactum, and let $f:K\to X$ be a reduced map from an ordered compactum $K$ onto $X$. The following conditions hold.
\begin{enumerate}
 \item \label{11} $\chi(K,k)\leq \chi(X,f(k))$ for every $k\in K$, whence $\chi(K)\leq \chi(X)$.
 \item \label{22} $c(K)= c(X)$, $d(K)=d(X)$, and $w(K)=w(X)$.
 \item \label{3} $\pi (K)=\pi (X)$.
 \item \label{4} $\pi\chi(K,k)\geq \pi\chi(X,f(k))$ for every $k\in K$, so that $\pi\chi (K)\geq \pi\chi (X)$. Furthermore, for every $x\in X$ there is a $k\in f^{-1}(x)$ such that $\pi\chi(K,k) = \pi\chi(X,x)$.
 \item \label{5} $t(K)\geq t(X)$.
 \item \label{6} $Nt(K) \leq Nt(X)$.
 \item \label{7} $\pi Nt(K)=\pi Nt(X)$.
 \item \label{8} For every $x\in X$, $\chi Nt(X,x) = \chi Nt(K,f^{-1}(x))$.
\end{enumerate}
\end{theorem}

\begin{proof}
We first note that the existence of an ordered compactum $K$ and a reduced map $f: K\to X$ is ensured by Lemma \ref{main_lm}.

\vspace{6pt}\noindent \textbf{(\ref{11})} Since $f$ is order-light, it follows that for any two distinct points $x,y\in K$ with $f(x)=f(y)$, there is an element $z\in K$ between $x$ and $y$ such that $f(z)\neq f(x)$. Let $k\in K$, and choose a local basis $\{V_\alpha:\alpha<\kappa\}$ for $X$ at $f(k)$, where $\kappa=\chi(X,f(k))$. For each $\alpha<\kappa$, let $U_\alpha$ be the order component of $f^{-1}(V_\alpha)$ containing $k$. 
Recall that the character $\chi$ and pseudo-character $\psi$ coincide on compact spaces.
It will be shown that $\bigcap_{\alpha<\kappa} U_\alpha=\{k\}$, so that $\chi(K,k)=\psi(K,k)\leq \kappa$. To this end, choose $m\in K$ such that $m\neq k$.

\begin{itemize}
 \item[\textit{Case 1.}] $f(k)\neq f(m)$. Let $\alpha<\kappa$ be such that $f(m)\notin V_\alpha$. Clearly $m\notin U_\alpha$.
 \item[\textit{Case 2.}] $f(k)=f(m)$. Then there is an $l$ between $k$ and $m$ such that $f(l)\neq f(k)$. Let $\alpha<\kappa$ be such that $f(l)\notin V_\alpha$. Since $U_\alpha$ is convex, $m\notin U_\alpha$.
\end{itemize}
In both cases, $m\notin U_\alpha$, and the result follows.

\vspace{6pt}\noindent \textbf{(\ref{22})} Since $f$ is irreducible, $c(X)=c(K)$. This is due to the fact that if $U$ is a non-empty open subset of $K$, then $f^\sharp(U)$ is a non-empty open subset of $X$.

To prove that $d(K)=d(X)$, choose a dense subset $D\subset X$, and for each $d\in D$ select a point $d'\in f^{-1}(d)$. Let $D'={\rm cl}(\{d':d\in D\})$. $D'=K$, for otherwise, the irreducibility of $f$ is contradicted. $K$ has thus a dense set of cardinality $|D|$.

Finally, let $\mathcal{B}$ be a base for $X$ of cardinality $w(X)$, and let $\mathcal{U}$ be the set of all order components of elements $f^{-1}(B)$, $B\in\mathcal{B}$. It is easy to see that $\mathcal{U}$ forms a base for $K$, since by (\ref{11}), it contains a local base for $K$ at $k$ for every point $k\in K$. Furthermore, $|\mathcal{U}|\leq |\mathcal{B}|\cdot c(X)=w(X)$. Hence, $w(K)\leq w(X)$. On the other hand, $f$ is a perfect map, and so does not increase weight.

\vspace{6pt}\noindent \textbf{(\ref{3})} Let $\mathcal{B}$ be a $\pi$-base for $X$ of cardinality $\pi (X)$, and let $U\subseteq K$ be a non-empty open subset of $K$. $f^\sharp(U)$ is then a non-empty open subset of $X$, and so it contains an element $V\in \mathcal{B}$. Since $f^\sharp(U)= X\setminus f(K\setminus U)$, $f^{-1}(V)\subseteq U$. This shows that $\mathcal{B}'=\{f^{-1}(V): V\in \mathcal{B}\}$ is a $\pi$-base for $K$, and so it follows that $\pi (K) \leq \pi (X)$. 
On the other hand, it is clear that since $f$ is irreducible, if $\mathcal{B}$ is a $\pi$-base for $K$, then $\{f^\sharp(U):U\in \mathcal{B}\}$ is a $\pi$-base for $X$, and so it follows that $\pi (K) \geq \pi (X)$.

\vspace{6pt}\noindent \textbf{(\ref{4})} Let $k\in K$. Let $\mathcal{B}$ be a local $\pi$-base for $K$ at $k$, and let $V$ be an open neighborhood of $f(k)$ in $X$. $f^{-1}(V)$ is an open neighborhood of $k$ in $K$, and so contains an element $U\in \mathcal{B}$. $V$ thus contains $f^\sharp(U)$. It follows that $\{f^\sharp(U):U\in \mathcal{B}\}$ is a local $\pi$-base for $X$ at $f(k)$. This shows that $\pi\chi(K,k)\geq \pi\chi(X,f(k))$.

Now let $x\in X$, and let $\mathcal{B}$ be a local $\pi$-base for $X$ at $x$. For each $U\in \mathcal{B}$, choose an order component $V_U$ of $f^{-1}(U)$, and let $\mathcal{B}'=\{V_U: U\in \mathcal{B}\}$. We will show that $\mathcal{B}'$ is a local $\pi$-base for $K$ at some point $k$ at which $f(k)=x$. Let $\mathcal{L}$ be a local base for $X$ at $x$. For every $L\in \mathcal{L}$ there is an element $B_L\in \mathcal{B}$ such that $\overline{B_L}\subset L$. 
Let $\mathcal{A}=\{B_L:L\in \mathcal{L}\}$, and let $\mathcal{A}'=\{\overline{V_U}: U\in \mathcal{A}\}$. By compactness, there is a subnet $S\subset\mathcal{A}'$ in the hyperspace $2^K$ which converges to a point $A\in 2^K$ which necessarily satisfies $f(A)=\{x\}$.
Since $f$ is order-light and all the members of $\mathcal{A}'$ are convex, $A$ is a singleton; say $A=\{k\}$. Obviously then, $\mathcal{B}'$ is a local $\pi$-base for $K$ at $k$, and it follows that $k$ satisfies $\pi\chi(K,k)\leq \pi\chi(X,f(k))$. Combining this with the first part of the proof, it follows that $\pi\chi(K,k) = \pi\chi(X,f(k))$.

\vspace{6pt}\noindent \textbf{(\ref{5})} This is a standard fact - it holds for any quotient map (see, e.g., \cite{Ju}).

\vspace{6pt}\noindent \textbf{(\ref{6})} 
Let $\mathcal{B}$ be a $\kappa^{op}$-like base for $X$, and let $\mathcal{U}$ be the family of all order components of all sets $f^{-1}(B)$, $B\in\mathcal{B}$. By (\ref{22}), $\mathcal{U}$ is a base for $K$. Suppose that there are an element $U\in \mathcal{U}$ and a collection $\mathcal{V}$ of $\kappa$-many elements of $\mathcal{U}$ such that $V\supseteq U$ for every $V\in \mathcal{V}$. Since distinct elements of $\mathcal{V}$ have non-empty intersection, it is clear that each pair of such elements of $\mathcal{V}$ are respectively order components of the inverse images of two distinct elements of $\mathcal{B}$. Now, $f^\sharp (U)$ is a non-empty open subset of $X$, and so contains an element $B\in \mathcal{B}$. $B$ is clearly then an element of $\mathcal{B}$ which is contained in $\kappa$-many elements of $\mathcal{B}$, which is a contradiction. Hence $\mathcal{U}$ is $\kappa^{op}$-like, and it follows that $Nt(K)\leq Nt(X)$. 

\vspace{6pt}\noindent \textbf{(\ref{7})} 
Let $\mathcal{B}$ be a $\kappa^{op}$-like $\pi$-base for $X$. By the proof of (3), $\mathcal{B}'=\{f^{-1}(V): V\in \mathcal{B}\}$ is a $\pi$-base for $K$, and by the proof of (6), $\mathcal{B}'$ is $\kappa^{op}$-like. Hence $\pi Nt(K)\leq \pi Nt(X)$. 

Assume on the other hand that $\mathcal{B}$ is a $\kappa^{op}$-like $\pi$-base for $K$.
By the proof of (3), the collection $\mathcal{U}=\{f^\sharp(U):U\in \mathcal{B}\}$ is a $\pi$-base for $X$.
Suppose that $\mathcal{U}$ is not $\kappa^{op}$-like. There are then an element $f^\sharp(U)\in \mathcal{U}$ and a collection $\mathcal{V}$ of $\kappa$-many elements of $\mathcal{U}$ such that $V\supseteq f^\sharp(U)$ for every $V\in \mathcal{V}$. Since $f^{-1}(f^\sharp(U))$ is non-empty and open, it contains an element $B\in \mathcal{B}$.
It is clear that $B$ must now be a subset of $\kappa$-many elements of $\mathcal{B}$, contradicting the fact that $\mathcal{B}$ is a $\kappa^{op}$-like.
$\mathcal{U}$ is then $\kappa^{op}$-like.
It follows that $\pi Nt(K)\geq \pi Nt(X)$.

\vspace{6pt}\noindent \textbf{(\ref{8})}
 This result holds in general for continuous surjections between compacta, and follows immediately from Lemma 2.22 of \cite{Mi1}.
\end{proof}

\begin{corollary}\label{cor1}
 Let $P$ be a property from $\{$first countable, ccc, separable, perfectly normal, second countable, has a countable $\pi$-base, has countable Noetherian type, has countable Noetherian $\pi$-type$\}$. Each space which is a continuous image of an ordered compactum and which satisfies $P$ is a continuous image of an ordered compactum which satisfies $P$.
\end{corollary}

\begin{proof}
 We only need to verify this for perfectly normality, in which case the result follows directly from the well-known observations that a perfectly normal compact space is ccc, and that a ccc ordered space is perfectly normal.
\end{proof}

\begin{example}\label{ex1}
 \textit{A reduced map may increase (local) character.} The one-point compactification $X$ of the uncountable discrete space of cardinality $\mathfrak{c}$ is a monotonically normal compactum which is not first countable, but is the reduced image of a first countable ordered compactum. Specifically, let $K$ be the 3-split interval; i.e. $[0,1]\times 3$ with the lexicographic order topology. Let $\sim$ be the equivalence relation which identifies all points in the top and bottom arrows of $K$; that is, for any $x, y\in [0,1]$ and any $i,j \in \{0,2\}$, $(x,i)\sim (y,j)$. It is clear that the quotient space $X = K/\sim$ is Hausdorff. The quotient map from $K$ onto $X$ is furthermore a reduced map. $X$ is thus a monotonically normal compactum which is not first countable, but is the reduced image of a first countable orderable compactum. This shows that the inequality may be strict in part (\ref{11}) of Theorem \ref{main_thm}.
\end{example}

Example \ref{ex1} is a discrete version of an example given in \cite{MP} to illustrate the same fact. Incidentally, it is also an example of a compact countably tight monotonically normal space with character $2^{\aleph_0}$, which contrasts sharply with the fact that tightness and character coincide on ordered spaces.

\begin{example}
 \textit{A reduced map may decrease (local) $\pi$-character.} Let $L\cup\{\infty\}$ be the one-point compactification of the long line $L$, and let $S$ be the quotient of $L\cup\{\infty\}$ which identifies the first and the last points of $L\cup\{\infty\}$. In analogy with the long line, $S$ can be called the \emph{long circle}. The quotient map, which we denote by $f$, is obviously a reduced map. $S$ is thus an example of a reduced image of a compact ordered space, $L\cup\{\infty\}$, where at one point, namely $\infty$, $\pi_\chi(L\cup\{\infty\}, \infty) > \pi_\chi(S, f(\infty))$. This shows that the inequalities in the first part of (\ref{4}) 
 of Theorem \ref{main_thm} may be strict.
\end{example}

\begin{example}
 \textit{A reduced map may decrease tightness.} Let $L\cup\{\infty\}$ be the one-point compactification of the long line $L$, and let $X$ be the quotient of $L\cup\{\infty\}$ which identifies all points of $\omega_1 \cup \{\infty\}$. $X$, which we call the $\omega_1$\emph{-flower}, is thus a reduced image of the compact ordered space $L\cup\{\infty\}$. Now, $t(L\cup\{\infty\})=\aleph_1$, whereas $t(X)=\aleph_0$. Thus, the inequality in (\ref{5}) of Theorem \ref{main_thm} may be strict.
\end{example}

\begin{example}
 \textit{A reduced map may increase Noetherian type.} Let $K=\omega_\omega \cup \{\omega_\omega\}$ be the one-point compactification of $\omega_\omega$. Also, let $X$ be the quotient of $K$ which has one equivalence class, the set $\{\omega_\alpha:\alpha\in\omega+1\}$, and denote the quotient map by $f$. Let $x=f(\omega)$. It is clear that $\chi Nt(X,x) = \aleph_\omega$, where as for any $k\in f^{-1}(x)$, $\chi Nt(K,k)< \aleph_\omega$. This shows that the first inequality in (\ref{6}) of Theorem \ref{main_thm} may be strict.
\end{example}

\begin{example}\label{setStage}
 \textit{A zero-dimensional monotonically normal compactum which contains no copy of $\omega_1$, no uncountable subspace of the real line, no uncountable subspace of the Sorgenfrey line, and no isolated points.} 
 \normalfont 
 The example here consists of copies of the space from Example \ref{ex1} built on top of each other to remove isolated points.
 Specifically, let $X_0$ be the space from Example \ref{ex1}.
 Assume now that for some $n< \omega$, the space $X_n$ has already been constructed.  
 For every isolated point $x$ of $X_n$, we ``glue'' a copy of $X_0$ to the space $X_n$ in such a way that this copy of $X_0$ has $x$ as its non-isolated point. The space thus obtained is denoted by $X_{n+1}$.
 For $n<m<\omega$, let $f^m_n:X_m\to X_n$ be the naturally defined projection, and let $X$ be the inverse limit of the inverse sequence $\{X_n, f^m_n, \omega\}$.
 We now construct an inverse sequence of compact ordered spaces with monotone and onto bonding maps whose inverse limit maps onto $X$.  
 Specifically, we start with the 3-split interval $K_0$, recalling that $X_0$ is an image of $K_0$, as described in Example \ref{ex1}. 
 Then, for $n > 0$, if the space $K_n$ has already been constructed, we let $K_{n+1}$ be the compact ordered space obtained from $K_n$ by replacing each isolated point in $K_n$ with a copy of the 3-split interval, equipping the resulting space with the natural order topology. 
 We then define all projections $\pi^m_n:K_m\to K_n$, $n<m<\omega$, in the obvious way. Since each factor space is compact and ordered and each bonding map is monotone, the inverse limit $K$ of the inverse sequence $\{K_n, \pi^m_n, \omega\}$ is a compact ordered space. 
 Moreover, it is easy to see that $K$ admits a map onto $X$.
 $X$ is thus a monotonically normal compact space. Furthermore, it is routine to check that $X$ has all the other requisite properties. 
 $X$ is an example of a universal pseudo-tree, as defined in \cite{Nps}.
\end{example}

It was proved in \cite{Fa} that SH implies that each monotonically normal compactum contains either an uncountable discrete subspace, an uncountable subspace of the real line, or an uncountable subspace of the Sorgenfrey line. Example \ref{setStage} could thus be seen as a preliminary construction setting the stage for the study of the 3-element basis conjecture for the class of uncountable subspaces of monotonically normal compacta.


\subsection{Push-forward}

We now utilize Theorem \ref{main_thm} 
to obtain results in the push-forward part of the scheme mentioned in the introduction.

The following is a folklore result.

\begin{proposition}\label{propo}
 Let $X$ be a space which is a continuous image of an ordered compactum. Then $X$ is ccc if and only if $X$ is perfectly normal.
\end{proposition}

\begin{proof}
 This follows from the proof of Corollary \ref{cor1} together with the fact that closed images of perfectly normal spaces are perfectly normal.
\end{proof}

Recall that a space is called \emph{non-Archimedean} iff it is T$_1$ and has a base any two members of which are either disjoint or comparable by inclusion. This notion was defined by Kurepa under the name ``spaces with ramified bases''.

\begin{theorem}\label{coro}
Each first countable monotonically normal compactum contains a dense non-Archimedean subspace.
\end{theorem}

\begin{proof}
 Let $X$ be a first countable monotonically normal compactum. 
 By (\ref{11}) of Theorem \ref{main_thm}, $X$ is an image of a first countable ordered compactum $K$. 
 Thus, by Lemma \ref{main_lm}, there are a first countable ordered compactum $K$ and a reduced map $f:K\to X$.
 By \cite{BT}, there is a dense subset $A$ of $K$ such that for each $x\in A$, $f^{-1}(f(x))=\{x\}$. 
 A theorem of Qiao and Tall states that any first countable ordered space has a dense non-Archimedean subspace \cite{QT}. 
 A minor adjustment of the proof of this theorem shows that the same result still holds for any subspace of a first countable ordered space. Hence, $A$ contains a dense non-Archimedean subspace $B$. To see that $Y=f(B)$ is a dense non-Archimedean subspace of $X$, it suffices to show that the bijection $f|_B:B\to Y$ is a closed map. Let $C\subseteq B$ be any closed subspace of $B$. Then $C=D\cap B$ for some closed subspace $D$ of $K$, and since $f$ is 1-1 on $B$,
 $$f|_B(C)=f(C)=f(D\cap B)=f(D)\cap f(B)=f(D)\cap Y\text{,}$$
 which is closed in $Y$.
\end{proof}

A \emph{Souslin} line is an ordered space which is ccc but not separable. 
Denote by \emph{SH} Souslin's Hypothesis that there are no Souslin lines.
It is known, for instance, that MA$_{\aleph_1}$ implies SH.

\begin{theorem}\label{theoro}
 Let $X$ be a ccc monotonically normal compactum. Then
\begin{enumerate}
  \item \label{111} $X$ contains a dense perfectly normal non-Archimedean (and orderable) subspace.
  \item \label{two} $X$ is either separable or contains a dense subspace which is a Souslin line.
\end{enumerate}
\end{theorem}

\begin{proof}
\noindent (\ref{111})
By Proposition \ref{propo}, $X$ is perfectly normal. A perfectly normal compact space is first countable. By Theorem \ref{coro}, $X$ contains a dense non-Archimedean subspace $D$ which is, of course, perfectly normal. A perfectly normal non-Archimedean space is orderable (\cite{Pu}).

\noindent (\ref{two})
Suppose $X$ is not separable. Recalling that cellularity and hereditary cellularity coincide in monotonically normal spaces \cite{Os}, it follows that $D$ from (\ref{111}) is orderable, ccc, and not separable. That is, $D$ is a Souslin line.
\end{proof}

\begin{corollary}\label{kiri}
 Each monotonically normal compactum satisfies one of the following alternatives:
  \begin{enumerate}
   \item $X$ contains an uncountable discrete subspace.
   \item $X$ is separable.
   \item $X$ contains a dense subspace which is a Souslin line.
  \end{enumerate}
\end{corollary}

Corollary \ref{kiri} thus implies, assuming SH, that each monotonically normal compactum either contains an uncountable discrete subspace or is separable. The following corollary establishes a variant of SH.

\begin{corollary} \label{zongi}
 SH holds if and only if each ccc monotonically normal compactum is separable.
\end{corollary}
  
\begin{proof}
 Assume that there exists a ccc monotonically normal compactum $X$ that is not separable. By Corollary \ref{kiri}, $X$ contains a dense subspace which is a Souslin line, and so $\neg\,$SH holds.
 
 Assume, on the other hand, that $\neg$ SH holds. The existence of a Souslin line implies the existence of a Souslin line $S$ which contains no dense metrizable subspace. The Dedekind completion of $S$ is a ccc ordered compactum, hence monotonically normal, that is not separable.
\end{proof}

\begin{remark}
The first part of the proof of Corollary \ref{zongi} also follows from Theorem \ref{main_thm}. For, since $X$ is not separable, it is the image of a ccc and non-separable ordered compactum $K$. That is, $K$ is a Souslin line.
\end{remark}

With the results on the existence of dense metrizable subspaces of monotonically normal compacta in place, one may wonder whether, if in analogy with Rosenthal compacta (Theorem 8 of \cite{To3}), each monotonically normal compactum contains a dense set of G$_\delta$ points. 
This is however not possible, as the Dedekind completion of an $\eta_1$-set would be a counterexample (cf. Chapter 13 of \cite{GJ}).
Nevertheless, we can still prove the following.

\begin{proposition}
 Each compact monotonically normal space contains a dense set of points of countable $\pi$-character.
\end{proposition}

\begin{proof}
 Let $X$ be a compact monotonically normal space. Then $X$ is the reduced image of a compact ordered space $K$. $K$ itself contains a dense set $A$ all of whose elements have countable $\pi$-character; this is folklore (see, e.g., \cite{Spa} for a proof). Since $X$ is a reduced image of $K$, part (\ref{4}) of Theorem \ref{main_thm} implies that $f(A)$ is a dense subset of $X$ such that $\pi_\chi(X,x)=\aleph_0$ for each $x\in f(A)$.
\end{proof}

\begin{theorem}
 Let $X$ be a monotonically normal compactum. The following are equivalent.
 \begin{enumerate}
  \item \label{s1} $X$ is metrizable.
  \item \label{s2} $X$ has an $\omega^{op}$-like base.
  \item \label{s3} $X$ has an $\omega_1^{op}$-like base.
 \end{enumerate}
\end{theorem}

\begin{proof}
 (\ref{s1}) implies (\ref{s2}) is Theorem 3.1 of \cite{Mi2}. (\ref{s2}) implies (\ref{s3}) trivially. To prove that (\ref{s3}) implies (\ref{s1}), assume that $X$ is a monotonically normal compactum that has an $\omega_1^{op}$-like base, and let $K$ be an ordered compactum that admits a reduced map $f:K\to X$. By (\ref{6}) of Theorem \ref{main_thm}, $K$ has an $\omega_1^{op}$-like base. By Theorem 4.5 of \cite{Mi2}, $K$ is metrizable, and hence, so is $X$. 
\end{proof}

Recall that a space $X$ is said to be \emph{homogeneous} iff for all $x, y\in X$ there is a homeomorphism $h:X\to X$ such that $h(x)=y$. 

\begin{proposition}\label{sarimari}
 Let $X$ be an image of a homogeneous ordered compactum. Then $\chi Nt(X)=\omega$.
\end{proposition}

\begin{proof}
 Let $L$ be a homogeneous ordered compactum that admits a map $f:L\to X$. By Theorem 2.32 of \cite{Mi1}, $\chi_K Nt(L)=\omega$, where $\chi_K Nt(L)$ is the supremum of the local Noetherian type of the compact subsets of $L$. By (\ref{8}) of Theorem \ref{main_thm}, for each $x\in X$, 
 $\chi Nt(X,x) = \chi Nt(L, f^{-1}(x))<\chi_K Nt(L)=\omega$.
 It follows that $\chi Nt(X)=\omega$. 
\end{proof}

Proposition \ref{sarimari} implies, for instance, that any space $X$ which is an image of the split interval satisfies $\chi Nt(X)=\omega$. 


\section{An embedding theorem}

For a space $X$ and a subset $Y$ of $X$, we denote by $K\,(X\setminus Y)$ the family of all components of $X\setminus Y$. Let $X$ be a locally connected continuum, and let $A$ be a non-empty and closed subset of $X$. $A$ is said to be a \emph{T-set} in $X$ if $bd(C)$ consists of exactly two points for each $C\in K\,(X\setminus A)$. If each $C\in K\,(X-A)$ is homeomorphic to $(0,1)$, $A$ is said to be a \textit{strong T-set} in $X$.

Let $\mathcal{A}$ be a family of subsets of a space $X$. $\mathcal{A}$ is said to be a \emph{null family} in $X$ if for each open cover $\mathcal{U}$ of $X$, the collection of all $A\in \mathcal{A}$ that are contained in no single $U \in \mathcal{U}$ is finite. When $X$ is compact, it can be shown that $\mathcal{A}$ is a null family in $X$ if and only if for every two disjoint closed subsets $G$ and $H$ of $X$, $A \cap G \neq \emptyset$ and $A \cap H \neq \emptyset$ can both hold for at most finitely many (at most countably many) $A\in \mathcal{A}$.

\begin{proposition}\label{illih}
 Let $f:X\to Y$ be a map from a space $X$ to a space $Y$. Let $\mathcal{A}$ be an (almost) null family of subsets of $X$. Then $\{f(A):A\in\mathcal{A}\}$ is an (almost) null family of subsets of $Y$.
\end{proposition}

\begin{proof}
 Let $\mathcal{V}$ be an open cover of $Y$. 
 $\mathcal{U} = \{f^{-1}(V): V\in\mathcal{V}\}$
 is then an open cover of $X$. 
 Since the collection of all $A\in \mathcal{A}$ that are contained in no single $U \in \mathcal{U}$ is finite (countable), it is clear that the collection of all $B\in \{f(A): A\in\mathcal{A}\}$ that are contained in no single $V \in \mathcal{V}$ is finite (countable).
\end{proof}

We note that Lemma \ref{main_lm} has been mostly used in the literature in conjunction with ``fillings'' of compact ordered spaces to obtain embedding theorems into locally connected continua. In particular, Nikiel proved in \cite{Ni} that each space which is a continuous image of a compact ordered space can be embedded as a strong T-set in a locally connected continuum which is a continuous image of an arc.

Since Nikiel's argument will be needed in the sequel, we sketch it presently. Let $X$ be a space which is a continuous image of a compact ordered space $K$ under a reduced map $f$. Embed $K$ into an arc $A$ by inserting a copy of the open unit interval $(0,1)$ into each jump of $K$. Consider the upper semi-continuous decomposition $\mathcal{A}$ of $A$ into sets $f^{-1}(x)$, $x \in X$, and one-point sets $\{a\}$, $a \in A\setminus K$. The quotient space $\widetilde{X}=A / \mathcal{A}$ is a locally connected continuum which is the continuous image of an arc, and in which the space $X$ embeds as a strong $T$-set.

\begin{example}
 \textit{Embedding Example \ref{ex1} as a strong T-set in a locally connected continuum that is the continuous image of an arc.}
 Let $X$ be the space from Example \ref{ex1}, and let $\widetilde{X}$ be the quotient of the lexicographically ordered square $[0,1]^2$ which identifies all points in $[0,1]\times 2$. It is clear that following the procedure in the sketch above will produce an embedding of $X$ as a strong T-set in $\widetilde{X}$.
\end{example}

The following theorem is an extension of Nikiel's result from \cite{Ni}.

\begin{theorem}\label{embedding}
Let $X$ be an image of an ordered compactum. Then there is a locally connected continuum $\widetilde{X}$ such that $\widetilde{X}$ is a continuous image of an arc, $X$ embeds as a strong T-set in $\widetilde{X}$, and the following conditions hold.
\begin{enumerate}
 \item \label{n1} For each non-isolated point $x\in X$, $\chi(\widetilde{X}, x)=\chi(X, x)$, and for each $x\in \widetilde{X}\setminus X$ or $x\in X$ isolated, $\chi(\widetilde{X}, x)=\omega$. Hence, that $\chi(\widetilde{X})=\chi(X)$.
 \item \label{n2} 
 $w(\widetilde{X}) = w(X)$.
 \item \label{n3} $\pi \chi (\widetilde{X}) \leq \pi \chi (X)$. 
 \item \label{n4} $t(\widetilde{X}) = t(X)$.
 \item \label{n5} For each non-isolated point $x\in X$, $\chi Nt(\widetilde{X}, x) \leq \chi Nt(X, x)$, and for each $x\in \widetilde{X}\setminus X$ or $x\in X$ isolated, $\chi Nt(\widetilde{X}, x)=\omega$. It follows that $\chi Nt(\widetilde{X}) \leq \chi Nt(X)$.
\end{enumerate}
\end{theorem}

\begin{proof} 
 To prove the theorem, we follow the aforementioned outline to embed $X$ as a strong T-set into a locally connected continuum $\widetilde{X}$ which is a continuous image of an arc $D$. 
 Recall that $D$ is obtained by inserting copies of $(0,1)$ into the jumps of a compact ordered space $K$ which admits a reduced map onto $X$.
 By construction therefore, $\widetilde{X}$ can be decomposed into $X$ and a disjoint family of copies of $(0,1)$.  
 Moreover, it is easy to see that the quotient map from $D$ onto $\widetilde{X}$ is reduced, and so we can make use of the results of Theorem \ref{main_thm}.
 
 \vspace{6pt}\noindent \textbf{(\ref{n1})} 
 Recall that a space is locally connected iff every component of every open set in the space is open. It is thus rather obvious that each $x\in \widetilde{X}\setminus X$ has countable character (each component of $\widetilde{X}$ is homeomorphic to $(0,1)$). We thus need only prove that for each $x\in X$, $\chi(\widetilde{X}, x)=\chi(X, x)$.
 
 Let $\mathcal{A}$ be the family of open arcs added to $X$ to obtain $\widetilde{X}$.
 For each $A\in \mathcal{A}$, let $I_A$ be the copy of $[0,1]$ in $\widetilde{X}$ obtained from the open arc $A$ by adding to it its endpoints, and let $\mathcal{I}=\{I_A:A\in\mathcal{A}\}$. 
 We first prove that $\mathcal{I}$ is a null family in $\widetilde{X}$. 
 Let $\mathcal{O}$ be the family of copies of $[0,1]$ in $D$ each of which is obtained from an open arc inserted in a jump of $K$ (in the process of building $D$) by adding to it its endpoints.
 By Proposition \ref{illih}, it suffices to show that $\mathcal{O}$ is a null family in $D$. 
 Suppose otherwise; i.e. that there are two disjoint closed subsets $G$ and $H$ of $D$, and a sequence 
 $\{O_i: i < \omega\} \subseteq \mathcal{O}$ 
 such that for all $i< \omega, O_i\cap G\neq \emptyset \neq O_i\cap H$. There is no loss of generality in assuming that this sequence is chosen in a monotone way, in the sense that a random sequence $\{o_i:o_i\in O_i\}_{i< \omega}$ is monotone. 
 For each $i\in \omega$, choose two points $k^0_i$ and $k^1_i$ such that $k^0_i \in O_i\cap G$ and $k^1_i\in O_i\cap H$. 
 It is readily verified that the limits of the sequences $\{k^0_i\}_{i<\omega}$ and $\{k^0_i\}_{i<\omega}$ can not be different. Since $G$ and $H$ are closed subsets of $K$, this contradicts the fact that $G\cap H=\emptyset$.

 Next, we show that the character at points of $X$ is preserved when we move to $\widetilde{X}$. 
 First, let $x\in X$ be an isolated point. Recalling that the map $f$ from the outline of Nikiel's argument is reduced, $x$ is an endpoint of at most two elements of $\mathcal{I}$. It follows that $\chi(\widetilde{X}, x)=\omega$.
 Let $x\in X$ be now non-isolated, $\kappa=\chi(X, x)$, and $\{U_\alpha:\alpha <\kappa\}$ be a family of open subsets of $X$ such that  $\bigcap_{\alpha<\kappa} U_\alpha=\{x\}$.
 Since $\mathcal{I}$ is a null family in $\widetilde{X}$, it follows that for each $\alpha < \kappa$, the subfamily $\mathcal{I}'\subset \mathcal{I}$ each of whose members intersects both subsets $\{x\}$ and $X\setminus U_\alpha$ of $\widetilde{X}$ is finite. 
 For each $\alpha<\kappa$ and each $i<\omega$, let 
 \begin{align*}
  V_\alpha^i = U_\alpha &\cup \bigcup\{A\in \mathcal{A}: (|I_A\cap U_\alpha|=2) \vee (x\notin I_A \wedge |I_A\cap U_\alpha|=1)\}\\
  &\cup \bigcup\{[x,\frac{1}{i+1}): (\exists I\in\mathcal{I})(\{x\}=U_\alpha\cap I)\}\text{,}
 \end{align*}
 where each $[x,\frac{1}{i})$ is the initial part of an arc from $\mathcal{I}$ that satisfies the relevant condition. Since there are at most finitely $I\in\mathcal{I}$ such that $\{x\}=U_\alpha \cap I$ (any such $I$ obviously belongs to $\mathcal{I}'$), it is easy to see that each $V_\alpha^i$ is open. 
 Clearly, moreover, $\bigcap_{\alpha<\kappa}\bigcap_{i<\omega} V_\alpha^i=\{x\}$, thereby finishing the proof.

 \vspace{6pt}\noindent \textbf{(\ref{n2})}
 Let $\mathcal{O}$ be the family of copies of $(0,1)$ inserted in jumps of $K$ to obtain $D$.
 Recall that the weight of an ordered space equals its order-density as an ordered set. 
 This means that already $w(K) \geq |\mathcal{O}|$. 
 Since each jump of $K$ is filled with a copy of $(0,1)$ in order to obtain $D$, it follows that $w(D)=w(K)$. 
 Therefore, by (\ref{22}) of Theorem \ref{main_thm}, $$w(\widetilde{X})=w(D)=w(K)=w(X)\text{.}$$
 
 \vspace{6pt}\noindent \textbf{(\ref{n3})}, \textbf{(\ref{n4})}
 These are easy observations and we omit the proofs.
  
 \vspace{6pt}\noindent \textbf{(\ref{n5})}
 Let $x\in X$ be a non-isolated point, and let $\mathcal{B}$ be a $\gamma^{op}$-like local base for $X$ at $x$.
 As in (\ref{n1}), let $\mathcal{A}$ be the family of all open arcs added to $X$ to obtain $\widetilde{X}$, and let $\mathcal{I}=\{I_A:A\in\mathcal{A}\}$ the family of copies $I_A$ of $[0,1]$ each of which is obtained from an open arc $A\in \mathcal{A}$ by adding to it its endpoints. 
 Also as in (\ref{n1}), for each element $B\in\mathcal{B}$ and each $i<\omega$, let $B^i\subset \widetilde{X}$ be defined by 
 \begin{align*}
  B^i = B &\cup \bigcup\{A\in \mathcal{A}: (|I_A\cap B|=2) \vee (x\notin I \wedge |I_A\cap B|=1)\}\\
  &\cup \bigcup\{[x,\frac{1}{i+1}): (\exists I\in\mathcal{I})(\{x\}=B\cap I), [x,\frac{1}{i+1})\subset I\}\text{.}
 \end{align*}
Now, let
 $\widetilde{\mathcal{B}} = \bigcup\{\{B^i:i< \omega\}: B\in \mathcal{B}\}\text{.}$
 By part (\ref{n1}), $\widetilde{\mathcal{B}}$ is a local base for $\widetilde{X}$ at $x$.
 Suppose that $\widetilde{B}_1$ and $\widetilde{B}_2$ are two elements of $\widetilde{\mathcal{B}}$ such that $\widetilde{B}_1\subset \widetilde{B}_2$. 
 First, $\widetilde{B}_1 = B_1^i$ and $\widetilde{B}_2 = B_2^j$ for some $B_1, B_2 \in \mathcal{B}$ and $i,j\in \omega$. 
 Assume now that $B_1\not\subset B_2$, and choose a point $y\in B_1\setminus B_2$. 
 Then $y\in \widetilde{B}_1\setminus\widetilde{B}_2$, which contradicts the fact that $\widetilde{B}_1\subset \widetilde{B}_2$.
 Hence, $B_1\subset B_2$.
 It thus follows that for any $x\in X$, $\chi Nt(\widetilde{X}, x) \leq \chi Nt(X, x)$. The other assertion, that $\chi Nt(\widetilde{X}, x)=\omega$ for each $x\in \widetilde{X}\setminus X$ or $x\in X$ isolated, is obvious.
\end{proof}

The next remark comments on the interplay between the order-theoretic and continuum-theoretic aspects of our study. Indeed, our main tool, Lemma \ref{main_lm}, can be viewed as the order-theoretic analogue of the monotone-light decomposition theorem from continuum theory.

\begin{remark} 
 \textit{An outline of how Theorem \ref{embedding} part (\ref{n1}) yields a proof of the fact that each first countable space which is a continuous image of an ordered compactum is a continuous image of a first countable ordered compactum (a consequence of Theorem \ref{main_thm} part (\ref{11})).} Let $X$ be a first countable space which is an image of a compact ordered space. $X$ embeds into a locally connected continuum $\widetilde{X}$ which is first countable and which is a continuous image of an arc. Nikiel's characterization of continuous images of arcs in \cite{Ni2} implies that $\widetilde{X}$ can be (strongly) approximated by finite dendrons, so that $\widetilde{X}$ is a continuous image of a dendron $T$ which is, in addition, first countable. We can show that $T$ is a continuous image of a first countable arc. Hence, $X$ is a continuous image of a first countable orderable compactum.
\end{remark}

\begin{corollary}
 Let $\kappa$ be a cardinal, and let $P$ be a property from $\{$countably tight, first countable, second countable, has $\pi$-character $\leq\kappa$, has local Noetherian character $\leq\kappa\}$. Each space $X$ which is a continuous image of an ordered compactum and which satisfies $P$ can be embedded as a strong T-set into a locally connected continuum $\widetilde{X}$ such that $\widetilde{X}$ satisfies $P$ and $\widetilde{X}$ is the continuous image of an arc.
\end{corollary}

The next example shows that Theorem \ref{embedding} can not be extended to cover cellularity and density.

\begin{example}
 \textit{Let $P$ be a property from $\{$ccc, separable, perfectly normal$\}$. A space which satisfies $P$ and is a continuous image of an ordered compactum can not always be embedded in a space which satisfies $P$ and which is the continuous image of an arc.}
 
 The split interval is a counterexample for any choice of $P$.
 To prove this, recall first that each continuous image of an arc is a locally connected continuum. 
 A theorem of Daniel in \cite{Da} states that a continuous image of an ordered compactum is metrizable iff it is separable and can be embedded as a G$_\delta$-set into a locally connected continuum. 
 
 Assume that the split interval $I_{(0,1)}$ can be embedded in a perfectly normal locally connected continuum. Since each closed subset of a perfectly normal space is a G$_\delta$-subset of the space, Daniel's theorem implies that the split interval is metrizable, which is a contradiction.
 
 Assume now that $I_{(0,1)}$ can be embedded in a ccc continuum $X$ which is the continuous image of an arc. By Proposition \ref{propo}, $X$ is perfectly normal. Hence, we arrive at a contradiction, as shown earlier.
 
 Finally, since a separable space is ccc, the previous analysis shows that $I_{(0,1)}$ can not be embedded in a separable continuum $X$ which is the continuous image of an arc.
\end{example}


\section{Applications to Boolean Algebras}

In this section, we prove results on cardinal invariants of Boolean algebras in the context of embeddings of pseudo-tree algebras in interval algebras. 

We first recall definitions and prove basic facts about interval algebras and pseudo-tree algebras.
Our presentation is geared towards giving a concise proof that the class of pseudo-tree algebras coincides with the class of subalgebras of interval algebras, before proceeding to applications.
One direction of this result was proved by Koppelberg and Monk in \cite{KM} (we prove this using a different technique), while the other was announced by Purisch in \cite{Purisch1}, with Purisch's proof not appearing in print apart from informal notes by J. D. Monk. We therefore give our proofs for completeness.

As usual, a space $X$ is called Boolean iff it is compact and zero-dimensional, and $Clop\,(X)$ denotes the collection of all closed and open subsets of $X$.
Also, for a Boolean algebra $B$, $Ult\,(B)$ denotes the set of all ultrafilters on $B$.
For unstated definitions, we refer the reader to the general references \cite{Koppelberg} and \cite{Monk} on Boolean algebras and their cardinal invariants, respectively.

\begin{definition}
 We say that a Boolean algebra $B$ is \emph{directly generated} by a subset $A\subset B$ iff $B$ is generated by $A$, and $A$ is closed under taking complements.
\end{definition}

\begin{lemma}\label{lemz}
 Let $B$ be a Boolean algebra, and let $s:B\to P(Ult(B))$ be the Stone map. If $B$ is directly generated by a subset $A\subset B$, then $s(A)$ forms a subbase for the topology on $Ult\,(B)$.
\end{lemma}

\begin{proof}
Recall that $s$ is a homomorphism from $B$ to $P(Ult(B))$, and that the topology on $Ult\,(B)$ is generated by $s(B)$.
Since $B$ is directly generated by $A$, each element $b\in B$ is a finite sum of finite products over $A$. 
Since $s$ is a homomorphism, it is clear then that $s(A)$ forms a subbase for the topology on $Ult\,(B)$.
\end{proof}

\subsection{Interval Algebras} 

Let $L$ be a linearly ordered set with first element $0_{L}$. Extend the linear order of $L$ to $L\cup\{\infty\}$, where $\infty$ is an element not contained in $L$, by setting $x<\infty$ for all $x\in L$. For $x,y\in L\cup\{\infty\}$, the set $[x,y)=\{z\in L\cup\{\infty\}: x\leq z < y\}$ is called the \emph{half-open interval} of $L$ determined by $x$ and $y$. The set of all finite unions of half-open intervals of $L$ is an algebra of sets over $L$. This Boolean algebra is called the \emph{interval algebra} of $L$, and denoted by $Intalg\,(L)$. $Intalg\,(L)$ can also be constructed for linear orders without first element by simply appending one.

It is not difficult to prove that a Boolean algebra is an interval algebra iff it has a set of generators which is linearly ordered under the natural partial ordering of the algebra. Indeed, this is how Mostowski and Tarski first introduced these algebras in \cite{MT}.

\begin{example}
The Sorgenfrey line has as basis the family 
$\mathcal{A} = Intalg\,(\mathbb{R})\text{.}$
The split interval, moreover, is the Stone space of the Boolean algebra $\mathcal{A}$.
\end{example}

\begin{proposition}\label{propsi}
 The dual space of an interval algebra is an ordered compactum.
\end{proposition}

\begin{proof}
 Let $Intalg\,(L)$ be the interval algebra of some linearly ordered set $L$.
 It is easy to see that $Intalg\,(L)$ is directly generated by a union of two chains; namely, $\{[l,\infty): l\in L\}$ and $\{[0_L, l): l\in L\}$. By Lemma \ref{lemz}, it follows that the dual space $K$ of $Intalg\,(L)$ is a compact Hausdorff space that has an open subbase which is the union of two chains. By \cite{vDW}, $K$ is orderable.
\end{proof}

Subalgebras of interval algebras need not be interval algebras; consider for instance the algebra of finite and co-finite subsets of $\omega_1$. This gives rise to the problem of characterizing the class of subalgebras of interval algebras. A first step towards resolving this problem is to notice that Proposition \ref{propsi} and standard facts on Stone duality yield the following result.

\begin{proposition}\label{props}
 The dual space of a subalgebra of an interval algebra is the continuous image of an ordered compactum.
\end{proposition}

\begin{proof}
Let $A$ be a subalgebra of an interval algebra $Intalg\,(L)$, and let $X$ be the dual space of $A$. By Proposition \ref{propsi}, the dual space of $Intalg\,(L)$ is an ordered compactum, which we denote by $K$. The embedding $e:A\to Intalg\,(L)$ is a 1-to-1 homomorphism of Boolean algebras, and so it induces a continuous and onto mapping $f:K\to X$.
\end{proof}

\subsection{Pseudo-tree Algebras}

For a partially ordered set $T$ and an element $t\in T$, let 
$t\downarrow = \{s\in T: s \leq t\}$ and
$t\uparrow = \{s\in T: s \geq t\}$.

A \emph{pseudo-tree} is a partially ordered set $(T,<)$ such that for each $t\in T$, the initial segment $t\downarrow$ of $T$ is linearly ordered. In the case that the initial segment of each element of $T$ is well-ordered, $T$ is called a \emph{tree}. Pseudo-trees are thus generalizations of both trees and linear orders.

For a pseudo-tree $T$, we define $Treealg\,(T)$, the \emph{pseudo-tree algebra of $T$}, to be the algebra of sets over $T$ generated by 
$\{t\uparrow: t\in T\}$.
If $T$ is a linearly ordered set, it is obvious that $Treealg\,(T)$ coincides with $Intalg\,(T)$.

Recall now that a family $\mathcal{C}$ of subsets of a set $X$ is called \emph{non-Archimedean} iff for all $c,d\in\mathcal{C}$, $c\subset d$, $d\subset c$, or $c\cap d = \emptyset$. By Theorem 2.3 of \cite{KM}, a Boolean algebra is a pseudo-tree algebra iff it has a non-Archimedean set of generators. 

The following proposition is based on J. D. Monk's development of Purisch's outline of his positive answer to the problem whether each subalgebra of an interval algebra is isomorphic to a pseudo-tree algebra \cite{Purisch1}. 

\begin{proposition}\label{props1}
 The dual algebra of a Boolean space which is a continuous image of an ordered compactum is isomorphic to a pseudo-tree algebra.
\end{proposition}

\begin{proof}
 Let $X$ be a Boolean space which is an image of an ordered compactum. By \cite{KM}, it suffices to show that $Clop\,(X)$ has a non-Archimedean set of generators.
By \cite{Nikiel5}, $X$ can be embedded as a strong T-set in a dendron $D$. Let $K\,(D\setminus X)$ be the collection of all components of $D\setminus X$, and fix $A\in K\,(D\setminus X)$. Fix a point $x_P$ in each component $P\in K\,(D\setminus X)$. By \cite{BNTT}, for each $P\in K\,(D\setminus X)$, $D\setminus \{x_{P}\}$ has exactly two components.
Let $U_A$ be a component of $D\setminus \{x_{A}\}$, and for each $N\in K\,(D\setminus X)$ different from $A$, let $U_N$ be the component of $D\setminus \{x_{N}\}$ which does not contain $x_A$. 
By \cite{BNTT}, the collection of all sets in
$$\mathcal{A} = \{U_P\cap X:P\in K\,(D\setminus X)\}$$
and their complements separates points in $X$.
$\mathcal{A}$ is thus a collection of closed and open subsets of $X$ generating $Clop\,(X)$.
We will show that $\mathcal{A}$ is non-Archimedean.
First, assume that for some $N\neq A$, $(U_N\cap X) \cap (U_A\cap X)\neq\emptyset$. 
Then $U_A\cup U_N$ is connected and does not contain $x_A$, which implies that $U_N\subset U_A$.
Now take two distinct components $N, P\in K\,(D\setminus X)\setminus \{A\}$, and assume that $(U_N\cap X) \cap (U_P\cap X)\neq \emptyset$. 
Clearly, $U_P\cup U_N$ is connected and does not contain $x_A$.
If $x_N\in U_P$, then $U_P\cup U_N\subset U_P$, and hence $U_N\subset U_P$. 
If $x_N\notin U_P$, then $U_P\cup U_N$ does not contain $x_N$, and so $U_P\cup U_N\subset U_N$, so that $U_P\subset U_N$.
\end{proof}

\subsection{The class of pseudo-tree algebras coincides with the class of subalgebras of interval algebras, and as such, is closed
under taking subalgebras}

In their treatise \cite{KM} on pseudo-tree algebras, Koppelberg and Monk used the compactness theorem of first order logic to produce an elegant proof that each pseudo-tree algebra is a subalgebra of an interval algebra (the result seems to have already been known to van Douwen - see, e.g., \cite{vD}). Koppelberg and Monk then asked whether the converse holds. The solution to their problem was announced by Purisch in \cite{Purisch1}. The surprising aspect of Purisch's approach is that it used results from continuum theory to prove a result on Boolean algebras. We prove Purisch's assertion in the following theorem.

\begin{theorem}\label{Pthm}
 Each subalgebra of an interval algebra is isomorphic to a pseudo-tree algebra.
\end{theorem}
 
\begin{proof}
 Let $B$ be a subalgebra of an interval algebra $Intalg\,(L)$, and let $X$ be the dual space of $B$. 
 By Stone duality, $B$ is isomorphic to $Clop\,(X)$.
 By Proposition \ref{props}, $X$ is an image of an ordered compactum. Proposition \ref{props1} now implies that $Clop\,(X)$ is isomorphic to a pseudo-tree algebra, which finishes the proof.
\end{proof}

We note that Theorem \ref{Pthm} was also given an algebraic proof by Heindorf in \cite{He}.

Next, we prove two propositions which, together with Propositions \ref{propsi} and \ref{props1}, complete the picture for the relationship between interval algebras and pseudo-tree algebras on the one hand, and compact ordered spaces and their continuous images on the other.

Recall that a family $\mathcal{C}$ of subsets of a set $X$ is called \emph{cross-free} iff for all $c, d\in \mathcal{C}$, $c\subset d$, $d\subset c$, $c\cap d = \emptyset$, or $c\cup d = X$.

\begin{proposition}\label{pripi}
 The dual algebra of an ordered Boolean space is isomorphic to an interval algebra.
\end{proposition}

\begin{proof}
 Let $K$ be an ordered Boolean space, and let $Clop\,(K)$ be the dual algebra of $K$. 
 By \cite{vDW}, $K$ has a subbase $B$ for closed and open sets which is the union of two chains; furthermore, we may choose $B$ so that each element of $B$ is either closed upwards or closed downwards. 
 $B$ is thus a direct set of generators for $Clop\,(K)$ which is the union of two chains, and so $Clop\,(K)$ has a set of generators which is linearly ordered. This implies that $Clop\,(K)$ is an interval algebra.
\end{proof}

\begin{proposition}\label{propzi}
 The dual space of a pseudo-tree algebra is the continuous image of a Boolean ordered compactum.
\end{proposition}

\begin{proof}
 Let $X$ be the dual space of a pseudo-tree algebra $Treealg\,(T)$. By \cite{KM}, $Treealg\,(T)$ has a non-Archimedean set of generators $B$. 
 Let $A$ be the closure of $B$ under taking complements. The set $A$ then directly generates $Treealg\,(T)$.
 Since $B$ is non-Archimedean, $A$ is cross-free.
 Let $s:Treealg\,(T)\to P(X)$ be the Stone map.
 Since $s$ is a homomorphism, it is straight-forward to verify $s(A)$ is cross-free.
 Now Lemma \ref{lemz}, in conjunction with Lemma 3 and Theorem 2 of \cite{Purisch2}, imply that $X$ is an image of a zero-dimensional compact ordered space, which finishes the proof.
\end{proof}

We notice now that Propositions \ref{pripi} and \ref{propzi} yield another proof of the Koppelberg-Monk result that each pseudo-tree algebra is isomorphic to a subalgebra of an interval algebra.

\begin{theorem}\label{propos}
 Each pseudo-tree algebra embeds into an interval algebra.
\end{theorem}

\begin{proof}
 Let $B$ be a pseudo-tree algebra. 
 By Proposition \ref{propzi}, the dual space $Ult\,(B)$ of $B$ is the continuous image of an ordered Boolean compactum $K$. 
 By Proposition \ref{pripi}, the dual algebra of $K$ is isomorphic to an interval algebra $A$. 
 The continuous map from $K$ onto $Ult\,(B)$ induces a homomorphic embedding of $B$ into $A$. 
\end{proof}

It thus follows, by Theorem \ref{Pthm} and Proposition \ref{propos}, that the class of pseudo-tree algebras coincides with the class of subalgebras of interval algebras. As such, the class of pseudo-tree algebras is closed under taking subalgebras.

\subsection{Applications}

We now present applications to Boolean algebras of the results obtained in Section \ref{secPullPush}. 

First, we recall the definitions of some cardinal invariants on Boolean algebras. 

\begin{definition}
 For Boolean algebras $A$ and $B$, and an ultrafilter $F$ on $A$,
 \begin{itemize}
  \item $B$ is said to be \emph{dense} in $A$ iff $B\subset A$ and for every $a\in A$ we can find $b\in B$ such that $b\subset a$;
  \item $\chi(A, F)=\min\{\kappa: F$ can be generated by $\kappa$-many elements of $A\}$ is the \emph{character} of $F$ in $A$;
  \item A subset $C\subset A\setminus\{0\}$ is said to be dense in $F$ if for all $a\in F$ there exists $b\in C$ such that $b\leq a$;
  \item $\pi\chi(A, F)=\min\{|C|:  C\subset A\setminus\{0\}$, $C$ is dense in $F\}$, the $\pi$\emph{-character} of $F$ in $A$;
  \item $t(A, F)=\min \{\kappa:$ if $C\subset Ult\,(A)$ and $F\subset \bigcup C$, then there is a subset $D\subset C$ of power at most $\kappa$ such that $F\subset\bigcup D\}$, the \emph{tightness} of $F$ in $A$.
 \end{itemize}
\end{definition}

\begin{definition}
 For a Boolean algebra $A$,
\begin{itemize}
 \item $c(A)= \sup \{|B|: B$ is a subset of $A$ whose elements are pairwise disjoint$\}$, the \emph{cellularity} of $A$;
 \item $d(A) = d(Ult\,(A))$, the \emph{topological density} of $A$;
 \item $\pi(A) = \min\{|B|: B$ is dense in $A\}$, the $\pi$\emph{-weight} or \emph{algebraic density} of $A$;
 \item $\chi(A)=\sup\{\chi(A,F):F$ is an ultrafilter on $A\}$, the \emph{character} of $A$;
 \item $t(A) = \sup\{t(A, F): F$ is an ultrafilter on $A\}$, the \emph{tightness} of $A$;
 \item $\pi Nt(A) = \min\{|B|: B$ is dense in $A$ and $B$ with the reverse natural ordering is $\kappa^{op}$-like$\}$, the \emph{Noetherian $\pi$-weight} or \emph{Noetherian algebraic density} of $A$.
\end{itemize}
\end{definition}

The next theorem further illustrates the interconnection between the topological and algebraic aspects of our study. 

\begin{theorem}\label{tree_card_inv}
 Each infinite pseudo-tree algebra $B$ embeds in an interval algebra $A$ such that the following conditions hold.
 \begin{enumerate}
  \item \label{aa} $|A|=|B|$.
  \item \label{bb} $d(A)=d(B)$ and $c(A)=c(B)$.
  \item \label{c} For every ultrafilter $F$ on $B$, $\chi(A,F')\leq \chi(B, F)$, where $F'$ is any ultrafilter on $A$ generated by $F$. It follows that $\chi(A)\leq \chi(B)$.  
  \item \label{d} $\pi(A) = \pi(B)$.
  \item \label{f} $\pi Nt(A)=\pi Nt(B)$.
 \end{enumerate}
\end{theorem}

\begin{proof}
 Let $B$ be a pseudo-tree algebra. By Theorem \ref{propos}, $B$ embeds into an interval algebra $A_0$. Let $g:Ult\,(A_0)\to Ult\,(B)$ be the onto mapping induced by the embedding of $B$ into $A_0$. By Proposition \ref{propsi}, $Ult\,(A_0)$ is a compact ordered space. 
 Therefore, Lemma \ref{main_lm} implies that there is an ordered compactum $K$ such that $K$ admits a reduced map $f:K\to Ult\,(B)$. $K$ is obtained from $Ult\,(A_0)$ by first taking a closed subspace $K_1\subset Ult\,(A_0)$, and then taking a continuous image $K$ of $K_1$. 
 The process of taking a closed subspace of $Ult\,(A_0)$ preserves zero-dimensionality. 
 $Clop\,(K_1)$ is thus, by Proposition \ref{pripi}, an interval algebra (which also happens to be a quotient algebra of $A_0$). 
 The problematic procedure is in taking a continuous image of $K_1$, as this might produce an ordered compactum $K$ which is not zero-dimensional, and we now deal with this constraint.
 
 Let $A_1=Clop\,(K_1)$ be the quotient algebra of $A_0$ associated with taking the closed (and zero-dimensional) subspace of $K_1$ of $Ult\,(A_0)$ to produce an irreducible map $f_1:K_1\to Ult\,(B)$.
 Let $\mathcal{C}$ be the collection of all non-degenerate convex components of all fibers of the map $f_1$. 
  The ordered compactum $K$ is obtained as a quotient of $K_1$, where the quotient map $q:K_1\to K$ collapses each element of $\mathcal{C}$ to one point. 
 If $K$ is not zero-dimensional, 
 since density coincides with hereditary density in monotonically normal spaces \cite{Ga},
 we can choose a dense subset $D\subset q(\bigcup\mathcal{C})$ such that $|D|\leq d(K)$ and split each point of $D$ into two points, as in Definition \ref{def_split}. 
 The obtained space $K_2$ is therefore a zero-dimensional compact ordered space that admits a map onto $K$. 
 Moreover, since the weight of an ordered space coincides with its order-density as an ordered set,
 $w(K_2)\leq w(K)+ |D| = w(K)$.
 Finally, since perfect maps between compact spaces do not increase weight, $w(K_2)=w(K)$.
 
 Now, since $f_1$ is irreducible, each element of $\mathcal{C}$ has empty interior. 
 Since any element $C\in \mathcal{C}$ is convex, $C$ consists of two points, which necessarily form a jump in $K_1$. 
 It is thus easy to see that $K_2$ can actually be obtained by factorizing the map $q$ through a zero-dimensional ordered compactum by choosing not collapse some jumps in $K_1$ that the map $q$ collapses. In particular, the only jumps collapsed are elements $C\in\mathcal{C}$ which satisfy $q(C)\notin D$.
 By construction therefore, $K_2$ is an image of $K_1$.
 
 Let $f_2:K_2\to K$ be the quotient map, as per the construction above. 
 It follows that the dual algebra $A$ of $K_2$ is an interval algebra which contains an isomorphic copy of $B$. 
 Moreover, by construction, $A$ is a subalgebra of a quotient algebra of $A_0$.
 
 \vspace{6pt}\noindent \textbf{(\ref{aa})} 
 We have established that $w(Ult\,(A))=w(K_2)=w(K)$. On the other hand, $w(K)=w(Ult\,(B))$ by part (\ref{22}) of Theorem \ref{main_thm}. It follows that we may assume that $|A|=|B|$.
 
 \vspace{6pt}\noindent \textbf{(\ref{bb})}
 By Theorem \ref{main_thm} part (\ref{22}), $d(K)=d(Ult\,(B))$ and $c(K) = c(Ult\,(B))$. Moreover, it is clear that $d(K_2)=d(K)$ and $c(K_2)=c(K)$. These facts imply that $d(A)=d(B)$ and $c(A) = c(B)$.
 
 \vspace{6pt}\noindent \textbf{(\ref{c})}
  It is easy to see that for any $k\in K_2$, $\chi(K_2,k) \leq \chi(K,f_2(k))$. 
 Moreover, by Theorem \ref{main_thm} part (\ref{11}), 
 $\chi(K,l)\leq \chi(Ult\,(B),f(l))$ for every $l\in K$.
 It follows that for any $k\in K_2$,
 $$\chi(K_2,k) \leq \chi(K,f_2(k))\leq \chi(Ult\,(B),f\circ f_2 (k))\text{.}$$
 Now, let $F$ be an ultrafilter on $B$, and let $F'$ be any ultrafilter on $A$ generated by $F$. Notice that
 $$\chi(A, F') = \chi(K_2, F') \leq \chi(Ult\,(B),f\circ f_2 (F')) = \chi(Ult\,(B), F) = \chi(B,F)\text{.}$$
 The result follows.
 
 \vspace{6pt}\noindent \textbf{(\ref{d})} 
 It is clear that $\pi(K_2) = \pi(K)$, and so by Theorem \ref{main_thm} part (\ref{3}),
 $$\pi(A) = \pi(K_2) = \pi(K) = \pi(Ult\,(B)) = \pi(B)\text{.}$$
 
 \vspace{6pt}\noindent \textbf{(\ref{f})} 
 The map $f_2$ was constructed so as to be irreducible. This implies, as in the proof of Theorem \ref{main_thm} part (\ref{7}), that $\pi Nt(K_2)=\pi Nt(K)=\pi Nt(Ult\,(B))$. It follows that $\pi Nt(A) = \pi Nt(B)$.
\end{proof}

\begin{definition}
 Let $B$ be a Boolean algebra.
 \begin{itemize}
  \item A set $C$ of non-zero elements of $B$ is said to be \emph{centered} if it satisfies the finite intersection property. $B$ is called $\kappa${-centered} if $B\setminus \{0\}$ is the union of $\kappa$-many centered sets (in case $\kappa=\omega$, we also say that $B$ is $\sigma$\emph{-centered}). When $B$ is infinite, 
$d(B) = \min\{\kappa: B$ is $\kappa$-centered$\}$.
  \item $B$ is called \emph{separable} iff $B$ has a countable dense subset iff $\pi(B)=\omega$.
 \end{itemize}
\end{definition}
 
The following then is a corollary of Theorem \ref{tree_card_inv}.

\begin{corollary} 
 Let $P$ be a property of Boolean algebras in $\{$countable, separable, $\sigma$-centered, ccc, each ultrafilter is countably generated, has countable Noetherian $\pi$-weight$\}$. 
 Each pseudo-tree algebra which satisfies $P$ is a subalgebra of an interval algebra which also satisfies $P$.
\end{corollary}


\section*{Acknowledgments}
 The author wishes to thank Prof.s A. Blaszczyk, W. Marciszewski, J. Nikiel and P. Nyikos for critical remarks and suggestions.

\bibliographystyle{amsplain}

\end{document}